\documentclass[preprint,12pt]{elsarticle}
\usepackage{amssymb,latexsym,amsmath,amsthm,amsfonts}
\usepackage{enumerate}

\newtheorem{theorem}{Theorem}[section]
\newtheorem{lemma}[theorem]{Lemma}
\newtheorem{corollary}[theorem]{Corollary}
\newtheorem{proposition}[theorem]{Proposition}

\theoremstyle{definition}
\newtheorem{definition}[theorem]{Definition}

\theoremstyle{remark}
\newtheorem{remark}[theorem]{Remark}

\numberwithin{equation}{section}

\newcommand{\BH}{B(\mathcal{H})}

\newcommand{\h}{\mathcal{H}}
\newcommand{\A}{\mathcal{A}}

\begin{document}
\begin{frontmatter}
\title{ Axioms required to get  Wold-Type decomposition}
\author{G. A. Bagheri-Bardi}\author{A. Elyaspour}
\address{Department of Mathematics, Faculty of Sciences, Persian Gulf University, Bushehr 75168, Iran}
\begin{abstract}
Three canonical   decompositions  concerning    commuting pair of isometries, power partial isometries, and contractions  are reassessed.
They have already been proved in von Neuamnn algebras. In the corresponding proofs,  both norm and weak operator topologies  are heavily involved. Ignoring topological structures, we give an algebraic approach to obtain them in the larger category of Baer $*$-rings.
\end{abstract}
\end{frontmatter}
\medskip

\section*{Introduction}\label{sec:0}
Let $\BH$ be the space of  bounded linear operators on a Hilbert space $\h$. Indeed $\BH$ is a particular type of von Neumann algebras.   von Neumann algebras enjoy algebraic, ordering, and topological structures  which make them highly complex. Ignoring the topological structure of von Neumann algebras,   Kaplansky in \cite{kaplansky} introduced a  larger category, called Baer $*$-rings. It  raised  the following serious  challenge.\medskip\\
{\bf Question.} {\it What  fundamental results in von Neumann algebras can be extended to the category of Baer $*$-rings?}\medskip\\
Any result in  Baer $*$-rings should be representation free i.e, independent of the underlying Hilbert space. So, any achievement in this field will be significant  and worth considering.
In two successful efforts in this regard \cite{BEE1,BEE2}, three fundamental decompositions concerning operators in $\BH$, listed below,  have been algebraically accomplished so that one may extend them to  Baer $*$-rings.

\begin{itemize}
\item (Wold \cite{Wold}) \emph{Every isometric operator on  $\h$ is uniquely decomposed to the direct sum of a unitary and a unilteral shift operator.}
\item (Halmos-Wallen \cite{Halmos}) \emph{Every power partial isometric operator on $\h$ is uniquely   decomposed to a direct sum whose summands are  unitary operators, unilateral shifts, back-ward shifts, and truncated shifts.}
\item (Nagy-Fioas-Langer \cite{Nagy}) {ٍ\it Every contraction operator on $\h$ is uniquely decomposed to a direct sum whose summands are unitary and completely non-unitary operators.}
\end{itemize}

Extension of each of the above-mentioned decompositions to relevant commuting pairs in $\BH$ results in the emergence of some interestingly challenging problems. We mention some cases that  are going to be dealt with in this current discussion.
Based on the Wold decomposition, for a  given commuting pair of isometric operators $(T_1,T_2)$ on $\h$, it might appear that there exist fourfold orthogonal subspaces $\h_{uu}, \h_{uu}, \h_{uu},$ and $\h_{ss}$  with
\[ T_n=T_{n_{|_{\h_{uu}}}}\oplus T_{n_{|_{\h_{us}}}}\oplus T_{n_{|_{\h_{su}}}}\oplus T_{n_{|_{\h_{ss}}}},\]
where the summands  are all either unitary or unilateral shift operators.
That is not always the case. Slosincki  characterized those commuting pairs of isometric operators enjoying  such a decomposition  \cite{Slocinski}, known as \emph{Wold-Slocinski} decomposition. Two decades later,  using dilation theory, Popovici suggested a decomposition that worked for every commuting pair of isomeric operators \cite{Popovici}. Extension of Halmos-Wallen  [resp. Nagy-Fioas-Langer] decomposition to a commuting pair of power partial isometric  [resp. contractive] operators, which is rather straightforward, was accomplished by Burdak too \cite{Burdak1}.


Results of this  work are listed belw.
\begin{itemize}
\item[{R1)}] For a given commuting pair of isometries in a Baer $*$-ring, a comprehensive discussion is given to specify
pairs enjoying decomposition whose summands are either unitary or unilateral shift. It extends  Wold-Slocinski decomposition to Baer $*$-rings. Our characterizations  paint a    much clearer picture  than what is given in \cite{Slocinski}. A progress upon  a general commuting pairs of isometries is also given, which supports  what has been done in \cite{BEE1}.
\item[{R2)}]   Halmos-Wallen and Nagy-Fioas-Langer decompositions, concerning a pair of commuting  power partial isometric and contractive operators,  are extended to Baer $*$-rings.
\end{itemize}

This paper is organized as follows. In the next section,  all required facts of Baer $*$-rings are assembled to yield the main results. The main results are given in section two.


\section{Preliminaries}
Let $\A$ be a $*$-ring, that is, a ring which admits a linear involution.
An element $x\in\A$ is called a {\it projection} [resp. {\it  a partial isometry/ an isometry}] if $x$ is a selfadjoint idempotent [resp. $xx^*x=x$/ $x^*x=1$]. If  $x^*$ is an isometry, then we call $x$ a {\it co-isometry}. When $x$ is both an isometry and a co-isometry, we say that $x$ is a {\it unitary}. For a given  partial isometry $x$ both $x^*x$ and $xx^*$ are projections and called the initial and final projection of $x$ respectively. We say   $x$  is   positive, written $x\geq0$, if  $x=\sum_1^nx^*_ix_i$ for some  $x_1,..., x_n\in \A$. The set of positive elements of $\A$, denoted by $\A_+$, form a  cone.

The left [resp. right] annihilator of a subset $S$ of $\A$ is denoted with $L(S)$
[resp. $R(S)$]. A $*$-ring is a Baer $*$-ring  if the right annihilator of every  nonempty subset of $S$
is generated, as a right ideal, by a projection $p$, that
is, $R(S)=p\A$. In this case $\A$ is unital.
\textbf{Throughout we assume $\A$ is a Baer $*$-ring.}
Now,  we list required  facts in Baer *-rings  to conclude our assertions. 
For more details, we refer to  \cite{Berb1,Berb2}.

For a given projection $p\in \A$, the corner subring  $\A_p:=p\A p$ is a   Baer $*$-ring as well. The involution of $\A$ is proper that is $aa^*=0$ implies $a=0$.
For  given projections $p$ and $q$ in $\A$, we write   $p \leq q$  to say  $pq=p$. 
We recall that the infimum [resp. supremum] of a family of projections $\{p_i\}$ is the largest [resp. smallest] projection $p$
with $p\leq  p_i$ [resp. $p_i\leq  p$]. The fact that any family of projections of $\A$ attain its infimum and  supremum in $\A$ plays the key role here. We denote the left projection of $a$ by $[a]$   that is the  smallest projection satisfying in $[a]a=a$.
For every  $b\in \A$,  we have  $ab=0$ if and only if $a[b]=0$.


There are some notions and notations that are applied in our work, but they do not appear in \cite{Berb1,Berb2}. We highlight them.

Let $x$ be in $\A$. For a given projection $p\in \A$, we say that $p$ is $x$-invariant if $xp=pxp$. Obviously $p$ commutes with $x$ if and only if $p$ is invariant under both $x$ and $x^*$

Taking direct sum of a family of pairwise orthogonal projections  $Q=\{p_i\}_{i\in I}$   makes senses. Indeed,
\[\bigoplus_{i\in I}p_i=\sup_{I_0\subseteq I}\sum_{i\in I_0} p_{i}\]
where the supremum is taken over all finite subsets $I_0\subseteq I$. We say $Q$ forms a \emph{basis}
if $\bigoplus_{i\in I}p_i=1$. When $Q=\{p_i\}_{i\in I}$ is a basis with $xp_i=p_ix$ for all $i\in I$,  then
$x$ is formally written  by the  direct sum   \[x=\bigoplus_{i\in I} xp_i.\] In this case, we say that $x$ is $Q$-\emph{block diagonalizable} whose diagonal blocks are  the summands $xp_i$'s.

Algebraic extension of the notion of  unilateral shift from the space of bounded linear operators  to  Baer *-rings was the milestone   in formation of the  discussion. It needs to be reviewed. For any isometry $x$ and projection $q$ in $\A$,  the following key identity holds,
\[[xq]=xqx^*.\]
This point yields that  for every isometry $x\in \A$, the following chain of projections are pairwise orthogonal.
\[1-[x]\curvearrowright [x(1-[x])]\curvearrowright\cdots[x^n(1-[x])]\curvearrowright\cdots.\]
In  concrete case, when $x$ is an isometric operator on a Hilbert space, $1-[x]$ is just the projection onto the kernel of the operator $x^*$. It is  called the wandering subspace of $x$.
It is proved in \cite{BEE1} that $x$ commutes with $p_s$  where $p_s=\sum_{n\geq0}[x^n(1-[x])]$  . We say that $x$ is a \emph{unilateral shift} provided that  $p_s=1$. If $x^*$ is a unilateral shift then $x$ is called a backward shift.


We end this section with a note that is interesting  per se. It determines a gap between von Neumann algebras and Baer $*$-rings.



Clearly  every projection is itself a positive element. If $p$ and $q$ are projections in $\BH$ then $q-p\geq0$ if and only if $p\leq  q$. In the following remark, we show this fact may no longer valid in Baer $*$-rings.

\begin{remark}\begin{enumerate}[{(1)}]
\item{} On the set of projections of $\A$, $p\leq  q$ clearly implies that $q-p\geq0$. The converse does not hold in general. To see this, let us consider the Baer $*$-ring $M_2(\mathbb{F}_3)$, the space of $2\times 2$ matrices over the field $\mathbb{F}_3=\{0,1,2\}$. For  projections
$p=\textrm{diag}\{1,0\}$ and $q=\textrm{diag}\{0,1\}$, we have that
\[q-p=\textrm{diag}\{2,1\}=p+p+q.\]
Thus $q-p\geq0$ but $p\neq qp=0$.
\item{} Let us suppose that for every $a$ in $\A$, the following occurs.
\[ a\geq0~ \textrm{and} ~-a\geq0 \Rightarrow a=0.\] 
Let $p$ and $q$ be projections with $q-p\geq0$. Thus, there exists finitely many elements $a_1,\cdots,a_n$ with $q-p=\sum a^*_ia_i$.
\begin{align*}
-(1-q)p(1-q)&=(1-q)(q-p)(1-q)
          \\&=(1-q)(\sum_1^n a^*_ia_i)(1-q)\geq 0.
\end{align*}
Note that $(1-q)p(1-q)\geq0$, so the assumption implies that $(1-q)p(1-q)=0$. Since the involution is proper, $p(1-q)=0$. Hence $p\leq  q$.
\end{enumerate}
\end{remark}


\section{Decomposition of a commuting pair of isometries.}


We refer to  Wold decomposition of an isometry in Baer $*$-rings because it delineates the outline well.

\begin{theorem}{\bf (Wold decomposition in Baer $*$-rings \cite{BEE1})}\label{T1.1}
Let $x$ be an isometry in $\A$. Then $x$ is uniquely decomposed to the direct sum of a unitary and  a unilateral shift. Indeed, we have
\[x=xp_u\oplus xp_s, \]
where $ p_u=\inf_{n\geq0} [x^n]$  and $p_s=\sum_{n\geq0}[x^n(1-[x])].$

\end{theorem}
We call the set of projections $\{p_u,p_s\}$  \emph{the Wold-basis} of $x$ in the sequel. The  theorem above  can be rewritten in this form: The Wold-basis $\{p_u,p_s\}$ diagonalizes   $x$
whose diagonal blocks  are either  unitary or unilateral shift. This motivation  causes  the following question.  \medskip \\
{\bf Question.} \emph{Let $(x_1,x_2)$ be a commuting pair of isometries in $\A$. Does there exist any basis $Q$  diagonalizing  both $x_1$ and $x_2$ simultaneously   such that the diagonal blocks are either unitary or unilateral shift elements?
}
\medskip \\
To address   this question, we put
\begin{itemize}
\item $p_{uu}$  [resp. $p_{ss}$] as the largest projection in $\A$ commuting with both $x_j$s such that the compression of $x_j$s to $p_{uu}$ [resp. $p_{ss}$] are unitary elements  [resp. unilateral shifts].
\item $p_{us}$ as the largest projection in $\A$ commuting with $x_j$s such that the compression of $x_1$ [resp. $x_2$] to $p_{us}$ is a unitary [resp. unilateral shift].
Similarly the projection $p_{su}$ is defined.
\end{itemize}
Let us denote,
\[\mathcal{Q}=\{p_{uu},p_{us},p_{su},p_{ss}\}.\]
At first glance,  $\mathcal{Q}$   appears to   fulfill  the basis mentioned in the  above question. Perusing  the following items in detail directs us to the goal.

\begin{itemize}
\item[{i)}] In general, $\mathcal{Q}$ may not  suffice  to diagonalize the pair $(x_1,x_2)$. 
\item[{ii)}] It will be fully characterized whenever  $\mathcal{Q}$ may form a basis in $\A$.

\end{itemize}

\begin{definition}
A commuting pair of isometries $(x_1,x_2)$  enjoy \emph{Wold-Slocinski}  decomposition if
\[x_i=\bigoplus_{\alpha,\beta\in \{u,s\}} x_ip_{\alpha\beta}, \mspace{20mu} (i=1,2).\]
\end{definition}
There is a spectrum of different situations resulting in  pairs of commuting isometries not enjoying Wold-Slocinski decomposition for which all the probabilities are feasible.  
The following theorem characterizes those commuting pairs  enjoying Wold-Slocinski decomposition.
\begin{remark}\label{us} For a given isometry $x\in \A$, let $q$ be a commuting projection with $x$. Suppose that $q\leq  p_{u}$ [resp. $q\leq  p_{s}$]. Then $xq$  is a unitary [resp. unilateral shift] in the corner subring $\A_{p_u}$ [resp. $\A_{p_s}$] (see \cite{BEE1}).
\end{remark}

\begin{theorem}\label{T1}
For a given  commuting pair of isometries $(x_1,x_2)$, let $\{p^{(i)}_u,p^{(i)}_s\}$ be the Wold-basis associated with $x_i$. The following statements are equivalent.
\begin{enumerate}
\item The pair $(x_1,x_2)$ enjoys the Wold-Slocinski decomposition.
\item  The following formulas hold.
\[p_{uu}=p^{(1)}_up^{(2)}_u,~ p_{us}=p^{(1)}_up^{(2)}_s,~ p_{su}=p^{(1)}_sp^{(2)}_u,~ p_{ss}=p^{(1)}_sp^{(2)}_s.\]
\item $x_1p^{(2)}_{\alpha}=p^{(2)}_{\alpha}x_1$ and $x_2p^{(1)}_{\alpha}=p^{(1)}_{\alpha}x_2$ where $\alpha\in\{u,s\}$.
\item $p^{(1)}_s$ is $x_2$-invariant and $p^{(2)}_s$ is $x_1$-invariant.
\item  $x_1p^{(2)}_{\alpha}=p^{(2)}_{\alpha}x_1$ and  $x_2$ commutes with either $p^{(1)}_{u}p^{(2)}_s$ or $p^{(1)}_{s}p^{(2)}_s$.
\item Both projections $p^{(1)}_s$ and $p^{(2)}_s$ are $x_1x_2$-invariant.
\end{enumerate}
\end{theorem}
\begin{proof}
$(1)\rightarrow(3)$: Suppose that $(x_1,x_2)$ enjoys Wold-Slocinski decomposition that is $p_{uu}+p_{us}+p_{su}+p_{ss}=1$. It yields the following formulas.
\[p^{(1)}_u=p_{uu}+p_{us},~ p^{(1)}_s=p_{su}+p_{ss},~ p^{(2)}_u=p_{uu}+p_{su}~ \textrm{and},~ p^{(2)}_s=p_{us}+p_{us}.  \]
Thus $x_2$ [resp. $x_1$] commutes with $p^{(1)}_\alpha$ [resp. $p^{(2)}_{\alpha}$], since both $x_j$s commutes with all $p_{\alpha\beta}$s.
\medskip \\
$(3)\to (2)$:
Note that the assumptions imply that $[x^n_1]p^{(2)}_{\alpha}=p^{(2)}_{\alpha}[x^n_1]$ and $[x^n_2]p^{(1)}_{\alpha}=p^{(1)}_{\alpha}[x^n_2]$ for every $n$. Applying \cite[Prop. 4.5]{Berb2}, we get the identities $p^{(1)}_{\alpha}p^{(2)}_{\beta}=p^{(2)}_{\beta}p^{(1)}_{\alpha}$ and $p^{(2)}_{\alpha}p^{(1)}_{\beta}=p^{(1)}_{\beta}p^{(2)}_{\alpha}$.  Thus, all products $p^{(1)}_{\alpha}p^{(2)}_{\beta}$ are projections. Note that  $p^{(1)}_{\alpha}p^{(2)}_{\beta}\leq  p^{(1)}_{\alpha}$  and $p^{(1)}_{\alpha}p^{(2)}_{\beta}\leq  p^{(2)}_{\beta}$. Remark (\ref{us}) confirms that $p^{(1)}_{\alpha}p^{(2)}_{\beta} \leq  p_{\alpha\beta}$. Moreover,
\[p_{\alpha\beta}\leq  \inf\{p^{(1)}_{\alpha},p^{(2)}_{\beta}\}=p^{(1)}_{\alpha}p^{(2)}_{\beta}.\]
Consequently  $p_{\alpha\beta}=p^{(1)}_{\alpha}p^{(2)}_{\beta}$.
\medskip\\
$(2)\to(1)$ A direct calculation leads us to obtain the assertion.
\[\sum_{\alpha,\beta\in \{u,s\}} p_{\alpha\beta}=\sum_{\alpha,\beta\in \{u,s\}} p^{(1)}_{\alpha}p^{(2)}_{\beta}=1.\]
\medskip\\
$(3)\leftrightarrow(4)$. The implication $(3)\to (4)$ is clear. For the converse note that
\[
  x_2p^{(1)}_u=p^{(1)}_u x_2  \Leftrightarrow \begin{cases}
x_2p^{(1)}_u=p^{(1)}_ux_2p^{(1)}_u\\
x_2p^{(1)}_s=p^{(1)}_sx_2p^{(1)}_s
\end{cases}.\]
Based on the assumption,  the identity $x_2p^{(1)}_u=p^{(1)}_ux_2p^{(1)}_u$ only needs to be verified.
\[x_2p^{(1)}_u x_2^*=x_2\inf x_1^nx_1^{*n}x_2^*=\inf x_1^nx_2x_2^*x_1^{*n}\leq  \inf x_1^nx_1^{*n}=p^{(1)}_u.\]
It means that $x_2p^{(1)}_u x_2^*=p^{(1)}_ux_2p^{(1)}_u x_2^*$. Multiplying $x_2$ from the right hand gets  $x_2p^{(1)}_u=p^{(1)}_ux_2p^{(1)}_u$. The other assertion is similarly obtained.  \medskip \\
$(4)\leftrightarrow(6)$: Suppose that the assumptions given in (4) hold. Then
\[x_1x_2p^{(1)}_s=x_1\big(p^{(1)}_sx_2p^{(1)}_s\big)=p^{(1)}_sx_1x_2p^{(1)}_s.\]
Similarly, one may check that $p^{(2)}_s$ are $x_1x_2$-invariant. To see the converse, let us multiply  $x_1^*$ from the left  to the identity $x_1x_2p^{(1)}_s=p^{(1)}_sx_1x_2p^{(1)}_s$. It gets $x_2p^{(1)}_s=p^{(1)}_sx_2p^{(1)}_s$ which means that $p^{(1)}_{s}$ is $x_2$-invariant. The other one is obtained similarly.\medskip \\
$(3)\leftrightarrow(5)$:
Obviously $(3)$ directs $(5)$. Now, suppose that $(5)$ holds. One may check directly that $p^{(1)}_up^{(2)}_u$ is  the largest projection commuting with  the isometry $x_1p^{(2)}_u$ such that the compression of  $x_1p^{(2)}_u$ to $p^{(1)}_up^{(2)}_u$  is a unitary element.
On the other hand,  $(x_1p^{(2)}_u,x_2p^{(2)}_u)$ is a doubly commuting pair of unitary elements.  Applying these two points yield that the projection $p^{(1)}_up^{(2)}_u$ commutes with  $x_2p^{(2)}_u$. It gets $x_2p^{(1)}_up^{(2)}_u=p^{(1)}_up^{(2)}_u x_2$. Combination this point with the assumption, we obtain that
\[\begin{cases}
x_2p^{(1)}_up^{(2)}_u=p^{(1)}_up^{(2)}_u x_2, \\
x_2p^{(1)}_up^{(2)}_s=p^{(1)}_up^{(2)}_s x_2.
\end{cases}\]
It is equivalent to say that  $x_2$ commutes with $p^{(1)}_{\alpha}$. It finishes the proof.
\end{proof}
\begin{corollary}
Let  $(x_1,x_2)$ be a commuting pair of isometries in $\A$.  Then $(x_1,x_2)$ enjoys Wold-Slocinski decomposition   if and only if both
$(x_1,x_1x_2)$  and $(x_2,x_1x_2)$ enjoy Wold-Slocinski  decompositions.
\end{corollary}

To find a remedy for commuting pairs not enjoying the Wold-Slocinski decomposition, let us put,
\[p_{ws}= 1-(p_{uu}+p_{us}+p_{su}).\]
The projection $p_{ws}$ is noticeably larger than $p_{ss}$ in general case. As  Example 3.9 in \cite{Popovici} shows,  $p_{ss}$ may be even vanish.
To extend   Wold-Slocinski  decomposition to general commuting pair of isometries, we need to find a characteristic property for the compression of the pair to $p_{ws}$. The next  arguments provide an approach to realize  this goal.
\begin{remark}
The compression of the  product $x_1x_2$ to $p_{ws}$ is a unilateral shift, since the unitary part of the isometry $x_1x_2$ is just $p_{uu}$ \cite{BEE1}. Despite of this illuminative description, the statue of the compression of $(x_1,x_2)$ to  $p_{ws}-p_{ss}$ is unclear yet. 

\end{remark}
\medskip
\begin{proposition}
Let $(x_1,x_2)$ be a commuting pair of isometries in $\A$. Let us consider
\[w^{\circ}_{su}=\inf (1-[x_2^{*n}x_1])~~~,~~~w^{\circ}_{us}=\inf (1-[x_1^{*n}x_2])\]
\begin{enumerate}[{(1)}]
\item The projection   $w^{\circ}_{us}$ [resp. $w^{\circ}_{su}$] is $x_1$-invariant [resp. $x_2$-invariant].
\item The compression of $x_1$ [resp. $x_2$] to $w^{\circ}_{us}$ [resp. $w^{\circ}_{su}$] is an isometry in $\A_{w^{\circ}_{us}}$ [resp. $\A_{w^{\circ}_{su}}$].
\item The compression of isometries $x_1w^{\circ}_{us}$,  $x_2w^{\circ}_{su}$, and $x_1x_2$ to $p_{ws}$ are all unilateral shifts. Indeed, $p_{ws}$ is the largest projection in which  these events simultaneously occur.
\end{enumerate}
\end{proposition}
\begin{proof} (1)
Clearly $x_1w^{\circ}_{us}x^*_1$ is a projection. Based on the first part of Proposition 3.8 given in  \cite{BEE1}, we observe  that  $x_1w^{\circ}_{us}x^*_1$ is majorized by $w^{\circ}_{us}$. Equivalently, $w^{\circ}_{us}x_1w^{\circ}_{us}x^*_1=x_1w^{\circ}_{us}x^*_1$. Multiplying $x_1$ from the right hand, we get the result
\[w^{\circ}_{us}x_1w^{\circ}_{us}=x_1w^{\circ}_{us}. \]
Similarly, it is proved that $w^{\circ}_{su}$ is $x_2$-invariant. \medskip\\
(2) By a  direct calculation one may see that both $x_1w^{\circ}_{us}$ and $x_2w^{\circ}_{su}$ are isometries. \medskip\\
(3) Applying the Wold decomposition theorem \cite[Theorem 2.4]{BEE1}, we have to check that
\[
\begin{cases}
p_{ws}\inf[(x_1x_2)^n]=\inf[p_{ws}(x_1x_2)^n]=0,\\
p_{ws}\inf[(x_1w^{\circ}_{us})^n]=\inf[(p_{ws}x_1w^{\circ}_{us})^n]=0,\\
p_{ws}\inf[(x_2w^{\circ}_{su})^n]=\inf[(p_{ws}x_2w^{\circ}_{su})^n]=0.
\end{cases}
\]
As earlier mentioned $p_{uu}=\inf[(x_1x_2)^n]$. Since $p_{uu}$ and $p_{ws}$ are orthogonal, the first one is got.
To obtain the next ones, we apply the following formula
\[(x_1w^{\circ}_{us})^n=x_1^nw^{\circ}_{us}.\]
which obtained by a direct calculation, applying the fact that $w^{\circ}_{us}$ is $x_1$ invariant. Hence, we may write
\[p_{ws}\inf[(x_1w^{\circ}_{us})^n]=p_{ws}\inf[x_1^nw^{\circ}_{us}].\]
Note that the projection $\inf[x_1^nw^{\circ}_{us}]$ is majorized by $p_{us}$. Thus,
\[p_{ws}\inf[x_1^nw^{\circ}_{us}] \leq  p_{ws}p_{us}=0.\]
Similarly one may see that  $p_{ws}\inf[x_2^nw^{\circ}_{su}]=0$. To complete the proof, let us suppose that $z$ is a projection commuting with both $x_1$ and $x_2$ such that the compression of isometries $x_1x_2$, $x_1w^{\circ}_{us}$ and, $x_2w^{\circ}_{su}$ to $z$ are all unilateral shifts.  Thus,
\[
\begin{cases}
z\inf[(x_1x_2)^n]=0,\\
z\inf[x_1^n w^{\circ}_{us}]=0,\\
z\inf[x_2^n w^{\circ}_{su}]=0.
\end{cases}
\]
We  show $z\leq  p_{ws}$. The first identity is equivalent to say $zp_{uu}=0$. It is  mentioned in \cite[Proposition 3.8]{BEE1} that
\[p_{us}=\sum[x_2^nw_{us}]\]  where $w_{us}=\inf[x_1^n w^{\circ}_{us}]$. It leads us to get
\[zp_{us}=\sum[x_2^nzw^{\circ}_{us}]=\sum[x_2^nz(\inf[x_1^n w^{\circ}_{us}])]=0.\]
Similarly  the identity $zp_{su}=0$ is proved. To sum up,
\[zp_{uu}=zp_{us}=zp_{su}=0.\]
Since $\mathcal{Q}$ forms a basis, then $z$ should be majorized by $p_{ws}$.
\end{proof}

\begin{definition}
Let $(x_1,x_2)$ be a commuting pair of isometries. It is  called a weak bi-shift if $x_2w^{\circ}_{su},x_1w^{\circ}_{us}$ and, $x_1x_2$ are all unilateral shifts.
\end{definition}

Now we conclude that;

\begin{theorem}\label{T2.1}
Let $(x_1,x_2)$ be a commuting pair of isometries in $\A$. Then $\{p_{uu},p_{us},p_{su},p_{ws}\}$ forms a basis in $\A$ making both $x_1$ and $x_2$ diagonalizable such that the diagonal blocks are  unitary, unilateral shift or weak bi-shift.
\[x_i=x_ip_{uu}\oplus x_ip_{us}\oplus x_ip_{su}\oplus x_ip_{ws} \hspace{1.5cm} (i=1,2).\]
\end{theorem}

\section{Decomposition of a pair of power partial isometries.}
This section is allocated  to   the extension of   Halmos-Wallen decomposition to a commuting pair of power partial isometries. To deal with this, let us review Halmos-Wallen theorem  concerning  a  power partial isometry $x\in \A$.


\begin{theorem}\label{hw}{\bf(Halmos-Wallen decomposition in Baer $*$-rings \cite{BEE2})}
Let $x$ be  power partial isometry in $\A$. We put,
\[
\begin{cases}
p_u=\inf_{n\geq0}\big\{[x^n],[x^{*n}]\big\},\\
p_s=(1-p_u)\inf_{n\geq 0}[x^{*n}],\\
p_b=(1-p_u)\inf_{n\geq 0}[x^n],\\
p_t=1-(p_u+p_s+p_b).
\end{cases}
\]
Then $Q_{HW}=\{p_u, p_s, p_b,p_t\}$ diagonalizes  $x$  whose diagonal blocks
$xp_u, xp_s, xp_b, xp_{t}$ are unitary, unilateral shift, backward shift and a direct sum of truncated shifts respectively.
\[x=xp_u\oplus xp_s\oplus xp_b\oplus xp_{t}.\]
\end{theorem}

\begin{definition}
Let $(x_1,x_2)$ be a commuting pair of power partial isometries. We say  $(x_1,x_2)$ enjoy Halmos-Wallen decomposition if there exists a basis $Q$ diagonalizes both $x_1$ and $x_2$  such that the compression of $x_i$'s to the projections contained in $Q$ are either unitary, unilateral shift, backward shift or truncated shift.
\end{definition}

Markedly, we should not  expect that every commuting pair of power partial isometries enjoy  Halmos-Wallen decomposition in general. Of course in some particular situations,
like doubly commuting case,  the decomposition will move forward  well. Indeed, let $Q^{j}_{HW}$ be the  Halmos-Wallen basis corresponding to $x_j$ ($j=1,2$) obtained in Theorem \ref{hw}. We have then,


\begin{proposition}\label{c}
Let $(x_1,x_2)$ be a doubly commuting pair of power partial isometries in $\A$. Then
\[Q^{(x_1,x_2)}_{HW}=\{pq: p\in Q^{1}_{HW}  ~,~q\in Q^{2}_{HW}\}\]
is a basis making both $x_1$ and $x_2$ diagonalize. The compression of $x_j$ to projections appeared in $Q^{(x_1,x_2)}_{HW}$ is  unitary, unilateral shift, backward-shift, or truncated shift.
\end{proposition}
For a commuting  pair of power partial isometries $(T_1,T_2)$ in $\BH$ that $T_1T_2$ remains a power partial isometry, Burduck proposed a decomposition \cite{Burdak1}. Axioms of Baer $*$-rings may still  afford it.

\begin{lemma}\label{A1}
Let $x$ be a power partial isometry. For every $n\in \mathbb{N}$
\[[x^*]x^{n-1}[x^{*n}]=x^{n-1}[x^{*n}].\]
\end{lemma}
\begin{proof}
For every $n\in \mathbb{N}$, note that   $[x^*][x^{n-1}]=[x^{n-1}][x^*]$. Thus,
\[[x^*]x^{n-1}x^{*n}=[x^*]x^{n-1}x^{*n-1}x^*=[x^*][x^{n-1}]x^*=[x^{n-1}][x^*]x^*=x^{n-1}x^{*n}.\]
 \end{proof}

\begin{lemma}\label{A2}
Let $(x_1,x_2)$ be a commuting pair of power partial isometries. Suppose that $x_1x_2$ is a power partial isometry. For every $n\in \mathbb{N}$
\begin{enumerate}
\item $[x_1^*]x_1^{n-1}[x_2^n]x_1^{*n}=x_1^{n-1}[x_2^n]x_1^{*n}.$
\item $[x_1^*[x_1^nx_2^n]]\leq[x_1^{n-1}x_2^{n-1}].$
\end{enumerate}

\end{lemma}
\begin{proof} (1) The assertion is equivalent to say
\begin{equation}\label{b1}
(1-[x_1^*])x_1^{n-1}[x_2^n]x_1^{*n}=0
\end{equation}
for every $n\in \mathbb{N}$.  Since  $x_1x_2$ is a power partial isometry, we have $[x_1^{*n}][x_2^n]=[x_2^n][x_1^{*n}]$ for every $n\in \mathbb{N}$. Using this point, Lemma\ref{A1} leads us to write
\begin{align*}
(1-[x_1^*])x_1^{n-1}[x_2^n]x_1^{*n}&=(1-[x_1^*])x_1^{n-1}[x_2^n]\underbrace{x_1^{*n}x_1^n}_{[x_1^{*n}]}x_1^{*n}
                                 \\&=(1-[x_1^*])x_1^{n-1}[x_1^{*n}][x_2^n]x_1^{*n}.
                                 \\&=(1-[x_1^*])[x_1^*]x_1^{n-1}[x_1^{*n}][x_2^n]x_1^{*n}=0.
\end{align*}
(2) First, we check the following identities.
\[
\begin{cases}
[x_1^{n-1}x_2^{n-1}]x_1^*[x_1^nx_2^n]=x_1^{n-1}[x_2^n]x_1^{*n},\\
x_1^*[x_1^nx_2^n]=[x_1^*]x_1^{n-1}[x_2^n]x_1^{*n}.
\end{cases}
\]
To prove the first identity, we apply again  $[x_1^{*n}][x_2^n]=[x_2^n][x_1^{*n}]$.
\begin{align*}
[x_1^{n-1}x_2^{n-1}]x_1^*[x_1^nx_2^n]&=x_1^{n-1}x_2^{n-1}x_2^{*n-1}x_1^{*n-1}x_1^*x_1^{n}x_2^{n}x_2^{*n}x_1^{*n}
                                   \\&=x_1^{n-1}[x_2^{n-1}][x_1^{*n}][x_2^{n}]x_1^{*n}
                                   \\&=x_1^{n-1}[x_2^{n-1}][x_2^{n}][x_1^{*n}]x_1^{*n}
                                   \\&=x_1^{n-1}[x_2^{n}]x_1^{*n}.
\end{align*}
As for the second one,
\begin{align*}
[x_1^*]x_1^{n-1}[x_2^n]x_1^{*n}&=x_1^*x_1x_1^{n-1}x_2^nx_2^{*n}x_1^{*n}=x_1^*x_1^nx_2^nx_2^{*n}x_1^{*n}
                             \\&=x_1^*[x_1^nx_2^n].
\end{align*}
Therefore, using Lemma \ref{A1}, we may write
\[[x_1^{n-1}x_2^{n-1}]x_1^*[x_1^nx_2^n]=x_1^*[x_1^nx_2^n],\]
which is equivalent to the assertion.
\end{proof}

\begin{theorem}\label{ppi}
Let $(x_1,x_2)$ be a commuting pair of power partial isometries in $\A$. Suppose that $x_1x_2$ is also a power partial isometry.  Let us put,
\[
\begin{cases}
p_u=\inf_{n\geq 0}\Big\{[x_1^nx_2^n], [x_1^{*n}x_2^{*n}]\Big\},\\
p_{is}=(1-p_u)\inf_{n\geq0}\big\{[x_1^{*n}x_2^{*n}]\big\},\\
p_{cis}=(1-p_u)\inf_{n\geq n}\big\{[x_1^{n}x_2^n]\big\},\\
p_{t}=1-\sup\{p_{is},p_{cis}\}.
\end{cases}
\]
Then $Q=\{p_u, p_{is}, p_{cis}, p_{t}\}$ is a basis in $\A$. The compression of $x_j$'s to  $p_u, p_{is},$ and  $p_{cis}$ are unitary,  isometry and,  co-isometry respectively. The compression of $x_1x_2$ to $p_t$ is a direct sum of truncated shifts.
\end{theorem}
\begin{proof} Let us put,
\[\tilde{p}_{is}=\inf_{n\geq0}\big\{[x_1^{*n}x_2^{*n}]~~~,~~~\tilde{p}_{cis}=\inf_{n\geq0}[x_1^nx_2^n].\]
\medskip \\
{\bf Claim.}{\it  We have
\begin{itemize}
\item[{(a)}] Both projections $\widetilde{p}_{is}$ and $\widetilde{p}_{cis}$ commute with $x_1$ and  $x_2$.
\item[{(b)}]  $\widetilde{p}_{is}$ [resp. $\widetilde{p}_{cis}$] is  the largest projection commuting with both $x_i$'s  such that $(\widetilde{p}_{is}x_1,\widetilde{p}_{is}x_2)$ [resp. ($\widetilde{p}_{cis}x_1,\widetilde{p}_{cis}x_2)$   form a pair of isometric elements in $\A_{\tilde{p}_{is}}$ [resp. co-isometry elements in $\A_{\tilde{p}_{cis}}$].
\end{itemize}}
{\it Proof of Claim.}  We get the assertions  concerning the projection $\widetilde{p}_{cis}$, the other one is similarly proved.

(a) Note that
\[
x_1\widetilde{p}_{cis}=\widetilde{p}_{cis}x_1\Leftrightarrow
\begin{cases}
\widetilde{p}_{cis}x_1\widetilde{p}_{cis}=x_1\widetilde{p}_{cis},\\
\widetilde{p}_{cis}x_1^*\widetilde{p}_{cis}=x_1^*\widetilde{p}_{cis}.
\end{cases}
\]
For every $n\geq1$,
\[[x_1\widetilde{p}_{cis}]\leq  [x_1[x_1^nx_2^n]]=[x_1x_1^nx_2^n]=[x_1^nx_2^n x_1]=[x_1^nx_2^n [x_1]]\leq  [x_1^nx_2^n].\]
It yields $[x_1\widetilde{p}_{cis}]\leq  \inf [x_1^{n}x_2^{n}]=p_{cis}$, equivalently  $\widetilde{p}_{cis}x_1\widetilde{p}_{cis}=x_1p_{cis}.$ Using the second part of Lemma \ref{A2},
\[[x_1^*\widetilde{p}_{cis}]\leq  [x_1^*[x_1^nx_2^n]]
                     \leq [x_1^{n-1}x_2^{n-1}].\]
Taking infimum,
\[[x_1^*\widetilde{p}_{cis}]\leq  \inf [x_1^{n-1}x_2^{n-1}]=\widetilde{p}_{cis}.\]
It means that $\widetilde{p}_{cis}x_1^*\widetilde{p}_{cis}=x_1^*\widetilde{p}_{cis}.$  Similarly, the result is obtained for $x_2$.

(b) Note that  $\widetilde{p}_{cis}\leq [x_1x_2]\leq [x_1]$. Thus, $\widetilde{p}_{cis}x_1x_1^*=\widetilde{p}_{cis}$ that means $\widetilde{p}_{cis}x_1$ is a co-isometry.
 It remains to prove that $\widetilde{p}_{cis}$ is the largest one. Suppose that $q$ is a projection commuting with both $x_i$'s such that $(qx_1,qx_2)$ forms a pair of co-isometry elements in $\A_q$. Note that,
 \[q[x_1^nx_2^n]=qx_1^nx_2^nx^{*n}_2 x^{*n}_1=x_1^nqx_2^nx^{*n}_2 x^{*n}_1=x_1^nq x^{*n}_1=q.\]
Consequently $q\leq  \inf_{n\geq 0} [x_1^{n}x_2^{n}]=\widetilde{p}_{cis}$.
\medskip \\
The claim guarantees that    $p_{u}, p_{is}$ and, $p_{cis}$ are the largest projections commuting with $x_j$'s such that the compression of the pair $(x_1,x_2)$ to projections $p_{u}, p_{is}$ and, $p_{cis}$ enjoy the desired properties.
\medskip \\
Let us apply  Theorem \cite[Theorem 3.12]{BEE2} for the power partial isometry $x_1x_2$. Then the     claim directs us to concluded that the compression of $x_1x_2$ to $p_t$ is a direct sum of truncated shifts in the corner subring $\A_{p_t}$.
\end{proof}
\begin{remark}
As we observed, the decomposition given in Theorem (\ref{ppi}) was accomplished when $x_1x_2$ is a power partial isometry.
As well as  being a doubly commuting case, if either $x_1$ or $x_2$ are isometry [resp. coisometry] then,  $x_1x_2$ will also be  a power partial isometry.
\end{remark}

Although the product of a commuting pair of partial isometry  is not necessarily a power partial isometry, the largest projection on which the case holds is computable. It determines the corner Baer $*$-ring in which Theorem \ref{ppi} always holds.
\begin{proposition}
Let $(x_1,x_2)$ be a pair of commuting power partial isometries. We put,
\[p=\sup\{q:  qx_i=x_iq~~(i=1,2),~q\leq  1-[[x_1^n][x_2^{*n}]-[x_2^{*n}][x_1^n]]\}\]
Then $p$ is the largest projection commuting with both $x_js$ such that the compression  $x_1x_2p$ is a power partial isometry in $\A_{p}$
\end{proposition}
\begin{proof}
We have that
\begin{align*}
p\leq  1-[x_1^nx_1^{*n}x_2^{*n}x_2^n-x_2^{*n}x_2^nx_1^nx_1^{*n}] &\Leftrightarrow p[x_1^nx_1^{*n}x_2^{*n}x_2^n-x_2^{*n}x_2^nx_1^nx_1^{*n}]=0
                      \\&\Leftrightarrow px_1^nx_1^{*n}x_2^{*n}x_2^n=px_2^{*n}x_2^nx_1^nx_1^{*n}.
\end{align*}
Note that $p$ commutes with both $x_js$. Thus, $px_1^nx_2^n$ is a power partial isometry. Trivially, $p$ is the largest one with these properties.
\end{proof}


\section{Decomposition of a pair of contractions}

The class of contractive operators in $\BH$ is much poorer than  power partial isometries. Thus, less information is expected to extract  from the decomposition concerning  contractive operators.

We recall that $x\in \A$ is a contraction if $1-xx^*$ is a positive element, equivalently $1-x^*x$ is positive (see \cite{BEE2}).

\begin{definition} (i) We say $\A$ has {\it antisymmetric property} provided that $\A_+$, the positive cone of $\A$, is proper; equivalently,  any finite sum of positive elements   $x_1+\cdots+x_n$ is non-zero unless  $x_i$s are all $0$.
\noindent (ii) We say that $\A$ is {\it smooth} if  $\A_+=\{x^*x\ :\ x\in \A\}$.
\end{definition}

To simplify, we put $x^n=x^{*-n}$ for negative integers.
\begin{theorem}\label{c}{\bf(Nagy-Fioas-Langer decomposition in Baer $*$-rings \cite{BEE2})}
Suppose that $\A$ is either smooth or antisymmetric. For a given contraction $x$,
let $p_u=\inf_{n\in \mathbb{Z}}q_n$ where $q_n= 1-[1-x^{*n}x^n]$.
Then $p_u$ is the largest projection commuting with $x$ such that the compression of $x$ to $p_u$ is a unitary element in the corner sub-ring $\A_{p_u}$.
\end{theorem}
We say that a  contraction $x$ is completely non-unitary if there is no non-zero projection $q$ commuting with $x$ such that $xq$ is a unitary in $\A_q$. If we put $p_c=1-p_u$ then the compression $xp_c$ will be completely non-unitary in $\A_{p_c}$.

\begin{definition}
Let $(x_1,x_2)$ be a commuting pair of contractions. It enjoys NFL-decomposition if there exists a basis $Q$  making both $x_1$ and $x_2$ diagonalize  such that the compression of $x_i$'s to projections contained in $Q$ are either  unitary or completely non-unitary.
\end{definition}

We call  $Q_{NFL}=\{p_u,p_c\}$  the NFL-basis of  $x$. The following point is obtained straightforward.

\begin{proposition}
Assume that $\A$ is either smooth or antisymmetric. Let  $(x_1,x_2)$ be  a doubly commuting pair of contractions in $\A$. It  enjoys NFL-decomposition. Indeed, if
we put,
\[Q_{NFL}^{(x_1,x_2)}=\{p^{(1)}_{\alpha}p^{(2)}_{\beta}: \alpha,\beta\in\{u,c\}\},\]
where $\{p^{(i)}_u,p^{(i)}_c\}$ is the NFL-basis corresponding to $x_i$.
Then $Q$ is a basis in $\A$. Moreover it makes both  $x_i$'s diagonasable whose diagonal blocks are either unitary or completely non-unitary.
\end{proposition}


\begin{theorem}
Suppose that $\A$ is either smooth or antisymmetric. Let $(x_1,x_2)$ be a commuting pair of contractions in $\A$.
We denote by $p_d$ the largest projection satisfying in the following relations.
\[ px_i=x_ip ~~,~~ p\leq 1-[x_2x_1^*-x_1^*x_2].\]
Then $p_d$ is the largest projection commuting with both $x_1$ and $x_2$ such that the pair $(x_1p_d,x_2p_d)$ is doubly commuting. Moreover $(x_1p_d,x_2p_d)$  enjoys NFL-decomposition.
\end{theorem}
\begin{proof}
Obviously  $p_d$ commutes with both $x_1$ and $x_2$. Let $p$ be a projection in $\A$. Note that,
\begin{align*}
p\leq  1-[x_2x_1^*-x_1^*x_2]
 &\Leftrightarrow p[x_2x_1^*-x_1^*x_2]=0\\
 &\Leftrightarrow p(x_2x_1^*-x_1^*x_2)=0.\\
\end{align*}
It implies that $(p_dx_1,p_dx_2)$ is doubly commuting. Applying Theorem (\ref{c}), we conclude that  $(p_dx_1,p_dx_2)$ enjoys NFL-decomposition.
\end{proof}



\bibliographystyle{amsplain}

\end{document}